\documentclass[12pt,reqno]{article}

\usepackage[utf8]{inputenc}
\usepackage{amssymb} 
\usepackage{amsmath}
\usepackage{amsthm}
\usepackage{amsfonts}
\usepackage{mathtools}
\usepackage{booktabs}
\usepackage{graphicx}
\usepackage[usenames]{color}
\usepackage[colorlinks=true,linkcolor=cyan,filecolor=webgreen,citecolor=cyan,urlcolor=cyan]{hyperref}

\definecolor{webgreen}{rgb}{0,.5,0}
\definecolor{webbrown}{rgb}{.6,0,0}

\usepackage{color}
\usepackage{fullpage}
\usepackage{float}

\theoremstyle{plain}
\newtheorem{theorem}{Theorem}
\newtheorem{corollary}[theorem]{Corollary}

\theoremstyle{definition}
\newtheorem{definition}[theorem]{Definition}
\newtheorem{example}[theorem]{Example}
\newtheorem*{burnside}{Burnside's lemma}

\newcommand{\fix}[1]{\mathsf{fix}(#1)}

\newcommand{\cS}{{\mathcal S}}
\newcommand{\cG}{{\mathcal G}}
\newcommand{\cF}{{\mathcal F}}

\newcommand{\inv}{{-1}}

\newcommand{\setst}[2]{ \left\{ #1 \mid #2 \right\} }
\newcommand{\set}[1]{\left\{#1\right\}} 

\newcommand{\seqnum}[1]{\href{https://oeis.org/#1}{\underline{#1}}}

\begin{document}
%
%
\begin{center}
\vskip 1cm{\LARGE\bf Orbits of Hamiltonian Paths and Cycles in Complete Graphs}
\vskip 1cm
\large 
Samuel Herman and Eirini Poimenidou\\
Division of Natural Sciences \\ 
New College of Florida\\
Sarasota, FL 34243\\
United States\\
\href{mailto:samuel.herman18@ncf.edu}{\tt samuel.herman18@ncf.edu} \\
\href{mailto:poimenidou@ncf.edu}{\tt poimenidou@ncf.edu} \\
\end{center}

\vskip .2 in


\begin{abstract}
We enumerate certain geometric equivalence classes of subgraphs induced by Hamiltonian paths and cycles in complete graphs. 
Said classes are orbits under the action of certain direct products of dihedral and cyclic groups on sets of strings representing subgraphs. 
These orbits are enumerated using Burnside's lemma. 
The technique used also provides an alternative proof of the formulae found by S. W. Golomb and L. R. Welch which give the number of distinct $n$-gons on fixed, regularly spaced vertices up to rotation and optionally reflection.
\end{abstract}


\section{Introduction}
All graphs in this paper are considered as their \textit{geometric realizations}, which are defined as follows: if a graph $G$ has $n$ vertices, its geometric realization is the figure obtained by first associating its vertices with $n$ regularly spaced points on a circle, then representing its edges as line segments between said points.
For example, the geometric realizations associated to the complete graphs $K_n$ for $3 \leq n \leq 6$ are
    \begin{center}
        \includegraphics[scale=0.9]{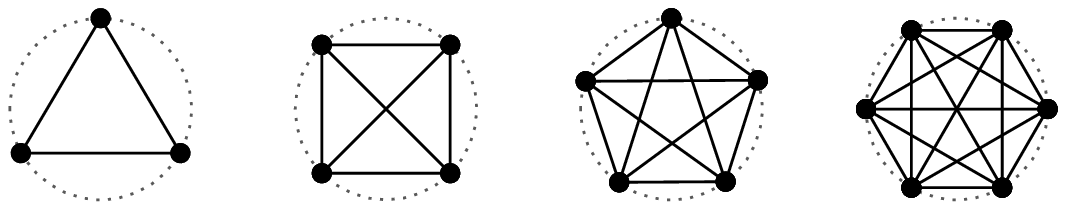}
    \end{center}

Recall that a \textit{Hamiltonian path} is a path in a graph which visits every vertex exactly once, and a \textit{Hamiltonian cycle} is a Hamiltonian path which is a cycle.
Any Hamiltonian path or cycle in a graph induces a subgraph whose vertex set is the same as the original graph, but whose edges consist of those traversed in the path or cycle; e.g.:
    \begin{center}
        \includegraphics{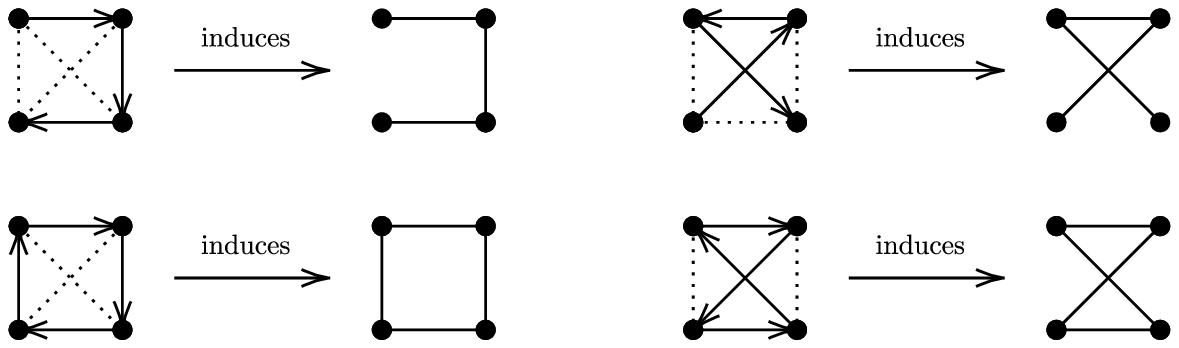}
    \end{center}

The investigation in this paper begins with a natural observation regarding the shapes of subgraphs induced by Hamiltonian paths in complete graphs. 
To see this observation for yourself, consider the set of subgraphs of $K_4$ induced by Hamiltonian paths which have an endpoint at the top left vertex:
    \begin{center}
        \includegraphics[]{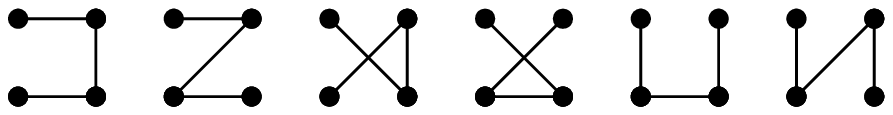}
    \end{center}
Notice that these subgraphs form one of just three distinct shapes---that is, any subgraph of $K_4$ induced by a Hamiltonian path is obtainable as a rotation or reflection of one of the following three graphs:
    \begin{center}
        \includegraphics[]{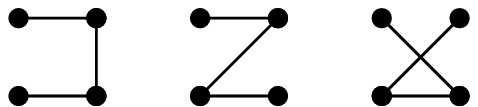}
    \end{center}
The analogous observation in the case of $K_5$ yields eight of these shapes:
    \begin{center}
        \includegraphics[]{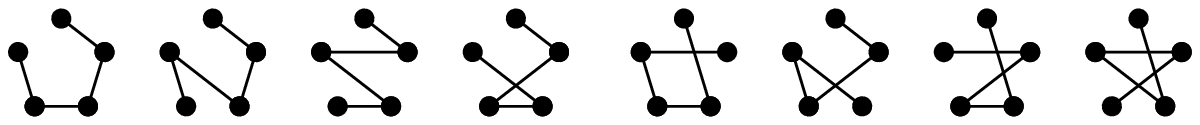}
    \end{center}
Furthermore, the analogous observations regarding subgraphs induced by Hamiltonian \textit{cycles} in $K_4$ and $K_5$ yield two and four of these shapes, respectively:
    \begin{center}
        \includegraphics[]{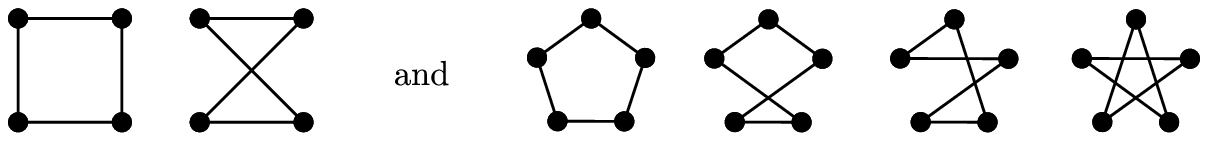}
    \end{center}

The natural inclination at this point is to ask whether there are formulae for enumerating the distinct shapes formed by the subgraphs induced by Hamiltonian paths or cycles in $K_n$. 
As it turns out, there are!
However, before we may show you, we must first make this question more precise.

\begin{definition}
Let $P_n$ and $C_n$ denote the sets of subgraphs of the complete graph $K_n$ which are induced by Hamiltonian paths or cycles, respectively.
Define the following equivalence relations on $P_n$ and $C_n$:
    \begin{enumerate}
        \item Two subgraphs $G_1, G_2$ are said to be \textit{similar}, denoted by $G_1 \equiv_S G_2$, if they are obtainable from one another by a rotation or reflection.
        \item Two subgraphs $G_1, G_2$ are said to be \textit{equivalent}, denoted by $G_1 \equiv_E G_2$, if they are obtainable from one another by a rotation (but \textit{not} a reflection).
    \end{enumerate}
\end{definition}

\begin{example}
\hfill
    \begin{center}
        \includegraphics[scale=1.1]{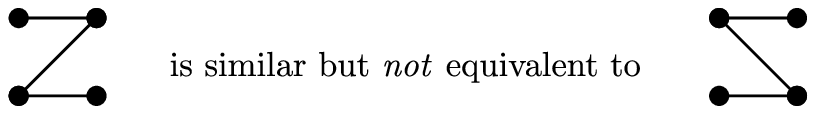}.
    \end{center}
\end{example}

This definition allows us to state our problem as one of enumerating the equivalence classes of either $P_n$ or $C_n$ under either $\equiv_S$ or $\equiv_E$, that is, we seek the sizes of the sets
    \[
        P_n / \equiv_S, \quad P_n / \equiv_E, \quad C_n / \equiv_S, \quad \text{and} \quad C_n / \equiv_E.
    \]
The sizes of these sets are given by
    \begin{align*}
    |P_n / \equiv_S| &= \frac{1}{4} \left[ (n-1)! +
        \begin{cases}
        (\frac{n}{2}+1)(n-2)!!], & \textnormal{if $n$ is even;}\\
        (n-1)!!], & \textnormal{if $n$ is odd.}
        \end{cases}
        \right],  \\
    |P_n / \equiv_E| &= \frac{1}{2} \left[ (n-1)! + 
        \begin{cases} 
        (n-2)!!], & \textnormal{if $n$ is even;} \\
        0, & \textnormal{if $n$ is odd.}
        \end{cases}
        \right],  \\
    |C_n / \equiv_S| &= \frac{1}{4n^{2}} \left[ \sum_{d | n} \left( \left(\phi\left(\frac{n}{d} \right) \right)^2  \left(\frac{n}{d}\right)^d  d! \right) +
        \begin{cases} 
        n!!\frac{n(n+6)}{4}, & \textnormal{if $n$ is even;}\\
        n^{2}(n-1)!!, & \textnormal{if $n$ is odd.}
        \end{cases}
        \right],  \\
    |C_n / \equiv_E| &= \frac{1}{2n^{2}} \left[ \sum_{d | n} \left( \left(\phi\left(\frac{n}{d} \right) \right)^2  \left(\frac{n}{d}\right)^d  d! \right) +
        \begin{cases} 
        \frac{n}{2}n!!, & \textnormal{if $n$ is even;} \\
        0, & \textnormal{if $n$ is odd.}
        \end{cases}
        \right],   
\end{align*}
where $n!!$ denotes the product of $n$ with every natural number less than $n$ of the same parity as $n$, i.e.,
    \[
        n!! = 
            \begin{cases}
                n (n-2) \cdots (2),  &\text{if $n$ is even;} \\
                n (n-2) \cdots (3) (1), &\text{if $n$ is odd.}
            \end{cases}
    \]
These formulae are proved in Theorems \ref{thmsp}, \ref{thmep}, \ref{thmsc}, and \ref{thmec} respectively. 
Further, illustrations of $P_n / \equiv_S$ and $C_n / \equiv_S$ for $3 \leq n \leq 6$ may be found in Figures \ref{fig:nspathsfig} and \ref{fig:nscyclesfig}, respectively.
Finally, Table \ref{tab:vals} gives the size of each of these sets for $3 \leq n \leq 10$.

    \begin{table}[]
        \centering
        \begin{tabular}{lcccc}
        \hline
        $n$ & $|P_n / \equiv_S|$ & $|P_n / \equiv_E|$ & $|C_n / \equiv_S|$ & $|C_n / \equiv_E|$ \\ \hline
        3   & 1                 & 1                 & 1                                & 1                 \\
        4   & 3                 & 4                 & 2                                & 2                 \\
        5   & 8                 & 12                & 4                                & 4                 \\
        6   & 38                & 64                & 12                               & 14                \\
        7   & 192               & 360               & 39                               & 54                \\
        8   & 1320              & 2544              & 202                              & 332               \\
        9   & 10176             & 20160             & 1219                             & 2246              \\
        10  & 91296             & 181632            & 9468                             & 18264             \\ \hline
        \end{tabular}
        \caption{Table of values for $3 \leq n \leq 10$.}
        \label{tab:vals}
    \end{table}
    

    \section{Setup}
    
We obtain these answers by converting the original problem into one of enumerating the orbits of a specific group action. 
These orbits are enumerated by way of Burnside's lemma, which is a standard tool in the theory of finite group actions.

    \begin{burnside} 
        Consider a group $G$ acting on a set $A$. 
        For each $g \in G$, let $\fix{g}$ denote the set of elements of $A$ which are fixed by $g$, i.e.,
            \[
                \fix{g} = \{ a \in A \; | \; g \cdot a = a \; \}.
            \]
        Let $A / G$ denote the set of orbits of this action. Then the number of orbits under the action of $G$ on $A$ is given by
            \[
                |A / G| = \frac{1}{|G|}\sum\limits_{g \in G} |\fix{g}|.
            \]
    \end{burnside}
    
We first define a set of strings which will represent the elements of $P_n$ or $C_n$.
\begin{definition}
Fix a labelling of the vertices of $K_n$ with the set $\mathbf{\bar{n}} = \{ 0, 1, \dots, n-1\}$, and let $X_n$ denote the set of $n$-length strings which are permutations of the elements of $\mathbf{\bar{n}}$, i.e., 
    \[
        X_n = \setst{x_0 x_1 \cdots x_{n-1}}{\text{$x_i \in \mathbf{\bar{n}}$ and $i \neq j \Rightarrow x_i \neq x_j$}}.
    \]
Note that $X_n$ has $n!$ elements.
\end{definition}

Next, we associate each string in $X_n$ with its interpretation as a graph in either $P_n$ or $C_n$ as follows. 

\begin{enumerate}
    \item $(X_n \longrightarrow P_n)$ Associate each string $(x_0 x_1 \cdots x_{n-1}) \in X_n$ with the subgraph of $K_n$ induced by the Hamiltonian path which traverses the vertices of $K_n$ in the order indicated by the string, i.e., the association is of the form
        \[
            (x_0 x_1 \cdots x_{n-1}) \longmapsto \langle x_0 \rightarrow x_1 \rightarrow \cdots \rightarrow x_{n-1} \rangle.
        \]
    This interpretation is illustrated by Figure \ref{fig:strings_paths}.
    Notice that both a string and its reversal are mapped to the same subgraph in $P_n$.
    \item $(X_n \longrightarrow C_n)$ Associate each string $(x_0 x_1 \cdots x_{n-1}) \in X_n$ with the subgraph of $K_n$ induced by the Hamiltonian cycle which traverses the vertices of $K_n$ in the order indicated by the string, i.e., the association is of the form
        \[
            (x_0 x_1 \cdots x_{n-1}) \longmapsto \langle x_0 \rightarrow x_1 \rightarrow \cdots \rightarrow x_{n-1} \rightarrow x_0 \rangle.
        \]
    This interpretation is illustrated by Figure \ref{fig:strings_cycles}.
    Notice that all cyclic permutations of a string as well as each of their reversals are mapped to the same subgraph in $C_n$.
\end{enumerate}

    \begin{figure}
        \centering
        \includegraphics{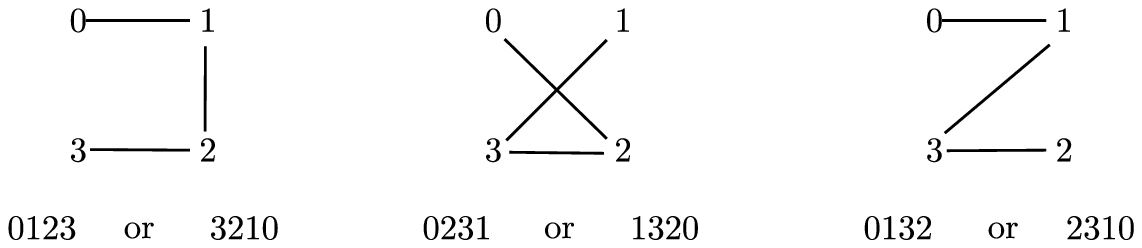}
        \caption{Interpretations of strings in $X_n$ as graphs in $P_n$.}
        \label{fig:strings_paths}
    \end{figure}
    
    \begin{figure}
        \centering
        \includegraphics{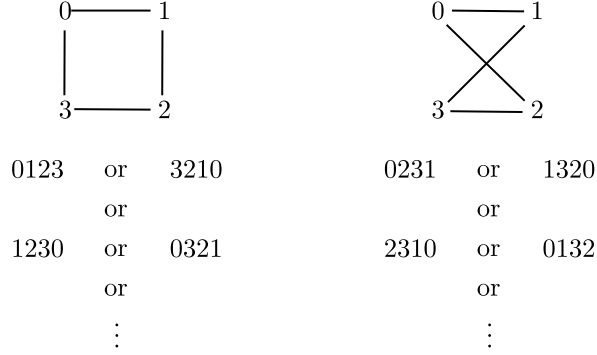}
        \caption{Interpretations of strings in $X_n$ as graphs in $C_n$.}
        \label{fig:strings_cycles}
    \end{figure}
    
Next, for each choice of $P_n$ or $C_n$ and $\equiv_S$ or $\equiv_E$, we define a group to act on $X_n$ such that the orbits of this action will coincide with the desired equivalence classes.
This group will be a direct product where the first coordinate acts purely on strings (i.e., an action of the first coordinate may only send strings to strings which have the same interpretation), while the second coordinate acts on a string's interpretation as a graph.

\begin{definition}
\hfill
\begin{enumerate}
    \item Define the following two groups which correspond to considering either $P_n$ or $C_n$:
        \begin{align*}
            \cS(P, n) & = \langle v \mid v^2 = 1 \rangle \\
                     & \text{and} \\
            \cS(C, n) & = \langle v, c \mid c^n = v^2 = 1, \; vcv = c^\inv \rangle.
        \end{align*}
    Note that $\cS(P,n)$ is isomorphic to the cyclic group of order 2, and $\cS(C,n)$ is isomorphic to the dihedral group of order $2n$.
    \item Define the following two groups which correspond to considering either $\equiv_S$ or $\equiv_E$:
        \begin{align*}
            \cG(\equiv_S, n) & = \langle r, s \mid r^n = s^2 = 1, \; srs = r^\inv \rangle \\
                   & \text{and} \\
            \cG(\equiv_E, n) & = \langle r \mid r^n = 1 \rangle.
        \end{align*}
    Note that $\cG(S,n)$ is isomorphic to the dihedral group of order $2n$, and $\cG(E,n)$ is isomorphic to the cyclic group of order $n$.
    \item Finally, given a choice of $\alpha = P_n, C_n$ and a choice of $\beta = \equiv_S, \equiv_E$, the acting group with respect to these choices is given by
        \[
            \Gamma(n, \alpha, \beta) = \cS(\alpha , n) \times \cG(\beta , n).
        \]
    For example, the acting group for $P_n$ under $\equiv_S$ is 
        \[
            \Gamma(n, P, \equiv_S) = \cS(P, n) \times \cG(\equiv_S, n).
        \]
\end{enumerate}

\end{definition}

\begin{definition}
The elements of $\Gamma(n, \alpha, \beta)$ act on strings in $X_n$ as follows:
    \begin{align*}
        (c, 1) \cdot (x_0 x_1 \cdots x_{n-1}) \; & = \; (x_1 \cdots x_{n-1} x_0), \\
        (v, 1) \cdot (x_0 x_1 \cdots x_{n-1}) \; & = \; (x_{n-1} \cdots x_1 x_0), \\
                                             & \text{and} \\
        (1, r) \cdot (x_0 x_1 \cdots x_{n-1}) \; & = \; \overline{(x_0 + 1)} \;\overline{(x_1 + 1)} \cdots \overline{(x_{n-1} + 1)}, \\
        (1, s) \cdot (x_0 x_1 \cdots x_{n-1}) \; & = \; (-x_0) (-x_1) \cdots (-x_{n-1}),
    \end{align*}
where $\overline{(x_i + 1)}$ denotes the sum $(x_i + 1)$ taken modulo $n$, and $(-x_i)$ denotes the (additive) inverse of $x_i$ modulo $n$.
The correct geometric interpretations of the actions of the second component are illustrated in Figure \ref{fig:action_strings}. 
\end{definition}

\begin{figure}
    \centering
    \includegraphics{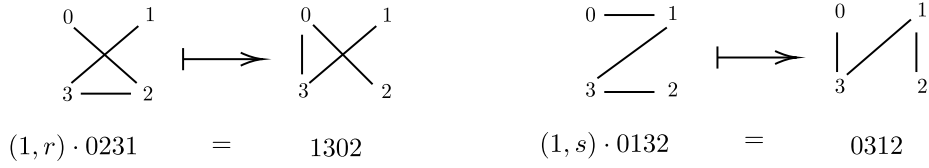}
    \caption{Geometric interpretation of the action of $\cG(S, n)$ on strings in $X_n$.}
    \label{fig:action_strings}
\end{figure}

This action has the following important properties:
    \begin{enumerate}
        \item Two strings in $X_n$ have the same interpretation as graphs in $P_n$ or $C_n$ if and only if they are contained in the same orbit under the action of $\cS(P, n)$ or $\cS(C, n)$ on $X_n$, respectively.
        \item If two strings in $X_n$ are contained in the same orbit under the action of $\cG(\equiv_S, n)$ or $\cG(\equiv_E, n)$, then their interpretations as graphs are similar or equivalent, respectively.
    \end{enumerate}
Considering these properties in the context of Burnside's lemma yields the observation that the equivalence classes of $P_n$ or $C_n$ under $\equiv_S$ or $\equiv_E$ correspond bijectively to the orbits of $X_n$ under the appropriate acting group.
In particular, we have the following:
    \begin{align*}
        |P_n / \equiv_S| &= |X_n / \Gamma(n, P, \equiv_S) | = \frac{1}{4n}\sum\limits_{g \in \Gamma(n, P, \equiv_S)} |\fix{g}|, \\
        |P_n / \equiv_E| &= |X_n / \Gamma(n, P, \equiv_E) | = \frac{1}{2n}\sum\limits_{g \in \Gamma(n, P, \equiv_E)} |\fix{g}|, \\
        |C_n / \equiv_S| &= |X_n / \Gamma(n, C, \equiv_S) | = \frac{1}{4n^2}\sum\limits_{g \in \Gamma(n, C, \equiv_S)} |\fix{g}|, \\
        |C_n / \equiv_E| &= |X_n / \Gamma(n, C, \equiv_E) | = \frac{1}{2n^2}\sum\limits_{g \in \Gamma(n, C, \equiv_E)} |\fix{g}|. \\
    \end{align*}


    \section{The path cases}


We begin by considering the cases involving $P_n$. 
This is because these cases turn out to be considerably simpler than those involving $C_n$, and thus they provide a suitable starting point for our investigation.
We first enumerate the classes of $P_n / \equiv_S$, and the size of $P_n / \equiv_E$ will follow as corollary.

 \begin{theorem} \label{thmsp}
    Let $n \geq 3$ be an integer. Then the number of equivalence classes of $P_n$ under $\equiv_S$ is given by
        \[
            |P_n / \equiv_S| = \frac{1}{4} \left[ (n-1)! +
                \begin{cases}
                    (\frac{n}{2}+1)(n-2)!!], & \textnormal{if $n$ is even;}\\
                    (n-1)!!], & \textnormal{if $n$ is odd.}
                \end{cases}
                \right]
        \]
    \end{theorem}

\begin{proof}

Notice that each element of $\Gamma(n, P, \equiv_S)$ may be expressed in exactly one of the forms
    \[
        (1, r^k), \; (1, sr^k), \; (v, r^k), \; (v, sr^k)
    \]
for some integer $0 \leq k \leq n-1$.
Considering this fact in the context of Burnside's lemma yields the observation that
    \[
        |X_n / \Gamma(n, P, \equiv_S)| = \frac{1}{4n} ( A_1 + A_2 + A_3 + A_4 ),
    \]
where
    \[
        A_1 = \sum_{k=0}^{n-1} |\fix{1, r^k}|, \quad A_2 = \sum_{k=0}^{n-1} |\fix{1, sr^k}|, \quad A_3 = \sum_{k=0}^{n-1} |\fix{v, r^k}|, \quad A_4 = \sum_{k=0}^{n-1} |\fix{v, sr^k}|.
    \]
We now evaluate each of these sums.

    \begin{enumerate}
    \item
Clearly $(1, r^k)$ will fix $(x_0 x_1 \cdots x_{n-1})$ only when $k=0$.
Hence $A_1 = |\fix{1, 1}| = n!$.
    \item
For each $0 \leq k \leq n-1$, if the action of $(1, sr^{k})$ fixes the string $(x_0 x_1 \cdots x_{n-1})$, then we must have $x_i \equiv_n -x_i - k$ for all $0 \leq i \leq n-1$, and so $2x_i + k \equiv_n 0$. 
But since there is some $x_j$ such that $x_j = 0$, it follows that $k=0$ and so $sr^{k} = s$, which will clearly fix no strings.
Thus $A_2 = 0$.
    \item 
Notice that $(v, r^k)$ will fix $(x_0 x_2 \cdots x_{n-1})$ if and only if $x_i \equiv_n x_{-(i+1)} + k$ and $x_{-(i+1)} \equiv_n x_i + k$ for all $0 \leq i \leq n-1$. 
This implies that $x_i \equiv_n x_i + 2k$ and thus that $2k \equiv_n n$. 
Hence $n$ must be even and, since $(v, 1)$ will clearly fix no strings, we have $k = n/2$. 

Now, $x_i \equiv_n x_{-(i+1)} + n/2$ implies that $x_i - x_{-(i+1)} \equiv_n n/2$. 
Hence, for each of the $(n/2)!$ bijections between the sets of compatible pairs of indices $\set{i, -(i+1)}$ and compatible pairs of labels $\set{x_i, x_{-(i+1)}}$:
    \begin{align*}
        P_{\mathsf{ind}} &= \set{ \set{0, n-1}, \set{1, n-2}, \dots, \set{\frac{n}{2}-1, \frac{n}{2}} } \\
        \updownarrow & \\
        P_{\mathsf{lab}} &= \set{ \set{0, \frac{n}{2}}, \set{1, \frac{n}{2} + 1}, \dots, \set{\frac{n}{2}-1, n-1} };
    \end{align*}
we obtain $2^{n/2}$ strings, for $2^{n/2} \cdot (n/2)! = n!!$ fixed strings in total; that is, we have
    \[ 
        A_3 = 
            \begin{cases}
                n!!, & \textnormal{if $n$ is even;}\\
                0, & \textnormal{if $n$ is odd.}
            \end{cases}
    \]
    \item
Notice that $(v, sr^{k})$ will fix $(x_0 x_1 \cdots x_{n-1})$ if and only if  $x_i \equiv_n -(x_{-(i+1)} + k)$ for all $0 \leq i \leq n-1$. 

First note that if $k$ is even, then there is some entry $x_j$ such that $x_j \equiv_n -k/2$.
Hence
    \[
        x_j \equiv_n -k/2 \equiv_n -x_{-(j+1)} - k,
    \]
which implies that $x_j \equiv_n x_{-(j+1)}$, and thus that $n$ must be odd.

We must consider the following cases.
    \begin{enumerate}
        \item
    If $n$ is even and $k$ is odd, then for each of $n/2$ possible values of $k$ and each of the $(n/2)!$ bijections between the sets of pairs of compatible indices and pairs of compatible labels:
    \begin{align*}
        P_{\mathsf{ind}} &= \set{ \set{0, n-1}, \set{1, n-2}, \dots, \set{\frac{n}{2}-1, \frac{n}{2}} } \\
        \updownarrow & \\
        P^{(k)}_{\mathsf{lab}} &= \set{ \set{n - \frac{k-1}{2}, n-\frac{k+1}{2}}, \dots, \set{n-1, n-k+1} };
    \end{align*}
    we have $2^{n/2}$ fixed strings, for a total of $(n/2) n!!$ fixed strings for this case.
        \item
    If $n$ is odd and $k$ is even, we have $2x_{\frac{n-1}{2}} \equiv_n -k$, and thus $x_{\frac{n-1}{2}} = n - \frac{k}{2}$.
    Hence our set of pairs of compatible indices is a pairing of the set $\mathbf{\bar{n}} - \set{\frac{n-1}{2}}$, and our set of pairs of compatible labels is a pairing of the set $\mathbf{\bar{n}} - \set{n - \frac{k}{2}}$.
    
    So for each even value of $k$ and each of the $(\frac{n-1}{2})!$ bijections between
    \begin{align*}
        P_{\mathsf{ind}} &= \set{ \set{0, n-1}, \set{1, n-2}, \dots, \set{\frac{n-3}{2}, \frac{n+1}{2}} } \\
        \updownarrow & \\
        P^{(k)}_{\mathsf{lab}} &= \bigg\{ \set{0, n-k}, \dots, \set{\frac{n-k-1}{2}, \frac{n-k+1}{2} }, \set{n-k+1, n-1},\\
        & \quad \dots, \set{n-1-\frac{k}{2}, n + 1 - \frac{k}{2}} \bigg\},
    \end{align*}
    we have $2^{\frac{n-1}{2}}$ fixed strings, yielding $(n-1)!!$ fixed strings in total for this case.
        \item
    If $n$ and $k$ are both odd, as above we have $2x_{\frac{n-1}{2}} \equiv_n -k$, and thus $x_{\frac{n-1}{2}} = \frac{n-k}{2}$.
    Hence our set of pairs of compatible indices is again a pairing of the set $\mathbf{\bar{n}} - \set{\frac{n-1}{2}}$, and our set of pairs of compatible labels is a pairing of the set $\mathbf{\bar{n}} - \set{\frac{n-k}{2}}$.
    
    So, as above, for each odd value of $k$ and each of the $(\frac{n-1}{2})!$ bijections between
    \begin{align*}
        P_{\mathsf{ind}} &= \set{ \set{0, n-1}, \set{1, n-2}, \dots, \set{\frac{n-3}{2}, \frac{n+1}{2}} } \\
        \updownarrow & \\
        P^{(k)}_{\mathsf{lab}} &= \bigg\{ \set{0, n-k}, \dots, \set{\frac{n-k}{2} - 1, \frac{n-k}{2} + 1 }, \set{n-k+1, n-1},\\
        & \quad \dots, \set{n-\frac{k-1}{2}, n - \frac{k+1}{2}} \bigg\},
    \end{align*}
    we have $2^{\frac{n-1}{2}}$ fixed strings, again yielding $(n-1)!!$ fixed strings in total for this case.
    Hence over all $n$ possible values of $k$ we have a total of $n(n-1)!!$ fixed odd-length strings.
    \end{enumerate}
Thus we obtain
    \[
        A_4 = 
            \begin{cases}
                (\frac{n}{2})n!!, & \textnormal{if $n$ is even;}\\
                n(n-1)!!, & \textnormal{if $n$ is odd.}
            \end{cases}
    \]
\end{enumerate}
Having evaluated each of these sums, the desired theorem now follows.
\end{proof}

Since $\Gamma(n, P, \equiv_E)$ is a subgroup of $\Gamma(n, P, \equiv_S)$, the number of equivalence classes of $P_n$ under $\equiv_E$ follows as an easy corollary.

    \begin{theorem} \label{thmep}
    Let $n \geq 3$ be a integer. Then the number of equivalence classes of $P_n$ under $\equiv_E$ is given by
        \[
            |P_n / \equiv_E| = \frac{1}{2} \left[ (n-1)! + 
                \begin{cases} 
                    (n-2)!!], & \textnormal{if $n$ is even;} \\
                    0, & \textnormal{if $n$ is odd.}
                \end{cases}
            \right]
        \]
    \end{theorem}
    


\begin{figure}
    \centering
    \includegraphics[scale=1.2]{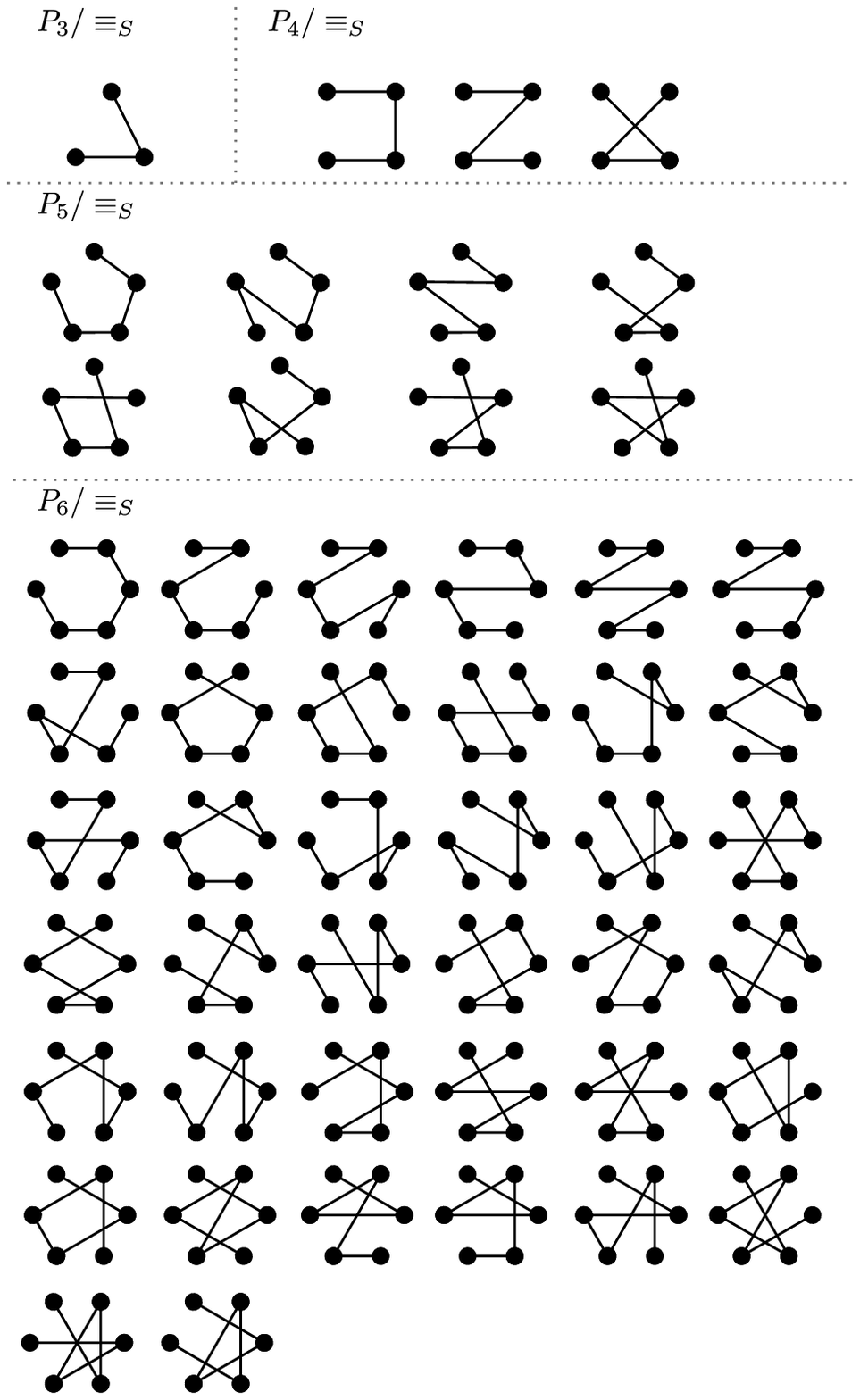}
    \caption{Representatives from each class of $P_n / \equiv_S$ for $3 \leq n \leq 6$.}
    \label{fig:nspathsfig}
\end{figure}

\begin{figure}
    \centering
    \includegraphics[scale=1.2]{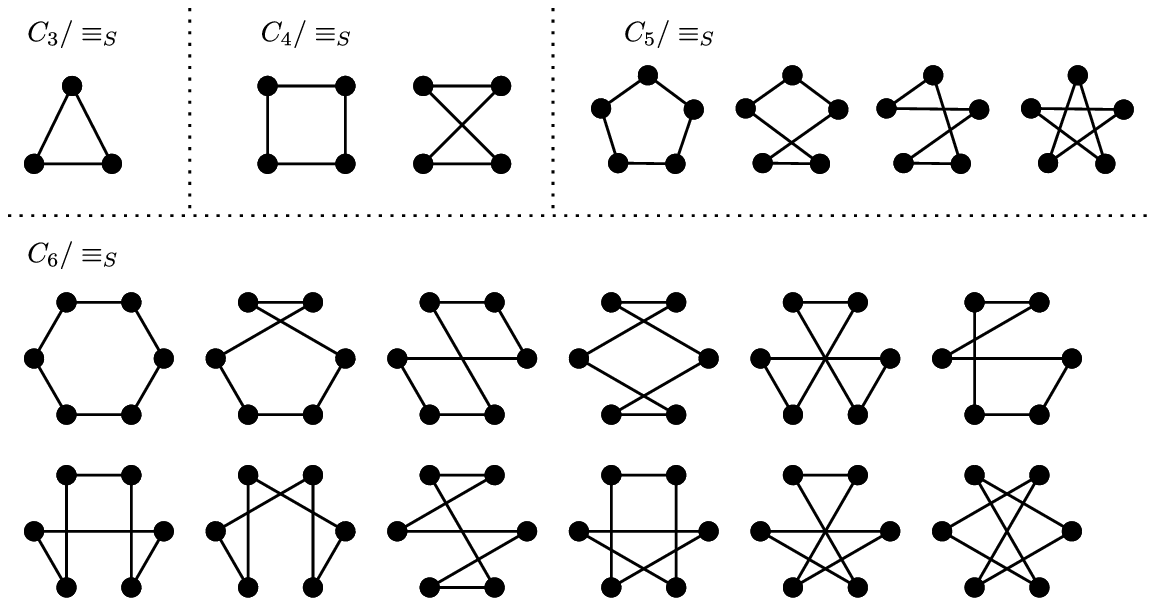}
    \caption{Representatives from each class of $C_n / \equiv_S$ for $3 \leq n \leq 6$.}
    \label{fig:nscyclesfig}
\end{figure}

    

    \section{The cycle cases}


We now turn to the more difficult problem of enumerating equivalence classes of $C_n$. 
As before, we begin by enumerating $C_n / \equiv_S$, and the size of $C_n / \equiv_E$ will follow as corollary.

\begin{theorem} \label{thmsc}
    Let $n \geq 3$ be an integer. Then the number of equivalence classes of $C_n$ under $\equiv_S$ is given by
        \[
            |C_n / \equiv_S| = \frac{1}{4n^{2}} \left[ \sum_{d | n} \left( \left(\phi\left(\frac{n}{d} \right) \right)^2  \left(\frac{n}{d}\right)^d  d! \right) +
                \begin{cases}
                n!!\frac{n(n+6)}{4}, & \textnormal{if $n$ is even;}\\
                n^{2}(n-1)!!, & \textnormal{if $n$ is odd.}
                \end{cases}
            \right]    
        \]
\end{theorem}

\begin{proof}
Notice that each element of $\Gamma(n, C, \equiv_S)$ may be expressed in exactly one of the forms
    \[
        (c^m, r^k), \; (c^m, sr^k), \; (c^mv, r^k), \; (c^mv, sr^k)
    \]
for some integers $0 \leq k, m \leq n-1$.
Considering this fact in the context of Burnside's lemma yields the observation that
    \[
        |X_n / \Gamma(n, C, \equiv_S)| = \frac{1}{4n^{2}} (B_1 + B_2 + B_3 + B_4),
    \]
where
    \begin{align*}
        & B_1 = \sum_{k,m=0}^{n-1} |\fix{c^m, r^k}|,  &B_2 = \sum_{k,m=0}^{n-1} |\fix{c^m, sr^k}|, \\\
        & B_3 = \sum_{k,m=0}^{n-1} |\fix{c^m v, r^k}|, &B_4 = \sum_{k,m=0}^{n-1} |\fix{c^m v, sr^k}|.
    \end{align*}
As before, we proceed to evaluate each of these sums.

\begin{enumerate}
    \item
    Notice that $(c^m, r^k)$ fixes $(x_0 x_1 \dots x_{n-1})$ if and only if $x_i \equiv_n x_{i+m} + k$ for all $0 \leq i \leq n-1$, and so $x_i \equiv_n x_{i+\ell m} + \ell k$ for all $0 \leq i \leq n-1$ and all $\ell \geq 0$.
    
    In particular, for $\ell = \frac{n}{\gcd(n,m)}$, we must have that $n \mid \ell m$ and so $x_i = x_{i + \ell m}$.
    Hence $x_i \equiv_n x_i + \ell k$ and $\ell k \equiv_n 0$; that is, $n \mid \ell k$ and thus $\gcd(n,m) \mid k$.
    Consequently, we have $\gcd(n,m) \mid \gcd(n,k)$.
    Similarly, for $\ell = \frac{n}{\gcd(n,k)}$, we have $n \mid \ell m$, implying that $\gcd(n,k) \mid m$ and hence $\gcd(n,k) \mid \gcd(n,m)$.
    
    We conclude that $\gcd(n,k) = \gcd(n,m) = d$ for some divisor $d$ of $n$, and therefore
        \begin{equation} \label{eq:fixeqdn}
            \sum_{\mathclap{0 \leq k, m < n}}|\fix{c^m, r^k}| = \sum_{d | n} \sum_{\substack{\gcd(k, n) = d \\ \gcd(m, n) = d}} |\fix{c^m, r^k}|.
        \end{equation}

    Now, fix some particular $k, m, d$ with $d = \gcd(n,k) = \gcd(n,m)$.
    We seek to determine the size of $\fix{c^m, r^k}$.
    Both of $r^k$ and $c^m$ have order $n/d$, and hence, if $(c^m, r^k)$ fixes $(x_0 x_1 \dots x_{n-1})$, then 
        \[
            x_i \equiv_n x_{i + \ell m} + \ell k \quad \text{for all $0 \leq i \leq n-1$ and $0 \leq \ell \leq \frac{n}{d}-1$.} 
        \]
    Hence, each choice of label $x_i$ determines the labels of all positions of the form $x_{i + \ell m}$ for $0 \leq \ell \leq \frac{n}{d}-1$.
    
    Note that, for $t \in \set{0, 1, \dots, n-1}$ with $\gcd(t, n) = d$ and $\alpha \in \set{0, \dots, d-1}$, the elements of the set
        \[
            \cF^{(t)}_\alpha = \setst{\alpha + \ell t}{0 \leq \ell \leq \frac{n}{d}-1}
        \]
    are all congruent to $\alpha$ modulo $d$ yet are all distinct modulo $n$ \cite{moser}.
    Setting $t = k, m$, we see that the set $\set{0, 1, \dots, n-1}$ may be partitioned in two different ways via $\Pi_m$ and $\Pi_k$, where
        \begin{align*}
            \Pi_m &= \set{ \set{0, m, \dots, (\frac{n}{d} -1)m}, \dots,  \set{d-1, d-1+m, \dots, 1 + d-(\frac{n}{d} -1)m }    }, \\
            &\text{and} \\
            \Pi_k &= \set{ \set{0, k, \dots, (\frac{n}{d} -1)k}, \dots,  \set{d-1, d-1+k, \dots, 1 + d-(\frac{n}{d} -1)k }    }.
        \end{align*}
    Consequently, the label of $x_i$ determines all labels with indices in the set $\cF^{(m)}_i$ such that all such labels are contained within a unique $\cF^{(k)}_{\alpha}$.
    In fact, we see that the labels of the initial substring $(x_0 x_1 \dots x_{d-1})$ completely determine the labelling of the rest of the string.
    
    Hence, for each of the $d!$ bijections between $\Pi_m$ and $\Pi_k$, each of the $(n/d)^d$ possible sets of choices of labelling $i \mapsto x_i \in \cF^{(k)}_{i}$ determines a unique fixed string.
    That is, for each particular valid $k, m, d$, we have
        \[
            |\fix{c^m, r^k}| = \left(\frac{n}{d}\right)^d d!.
        \]
    
    Combining this with (\ref{eq:fixeqdn}), we obtain
        \begin{align*}
        \begin{split} 
            B_1 = \sum_{\mathclap{k,m=1}}^{n-1}|\fix{c^m, r^k}| &= \sum_{d | n} \sum_{\substack{\gcd(k, n) = d \\ \gcd(m, n) = d}} |\fix{c^m, r^k}, \\
                      &= \sum_{d | n} \left(\frac{n}{d}\right)^d d! \sum_{\gcd(k, n) = d} \sum_{\gcd(m, n) = d} 1, \\
                      &= \sum_{d | n} \left(\phi\left(\frac{n}{d} \right) \cdot \phi\left(\frac{n}{d} \right) \cdot \left(\frac{n}{d}\right)^d \cdot d! \right),
        \end{split}
        \end{align*}
    where $\phi$ denotes Euler's totient function.
    \item
    Notice that $(c^m, sr^k)$ fixes $(x_0 x_1 \dots x_{n-1})$ if and only if $-x_{i+m} - k \equiv_n x_i$ for all $0 \leq i \leq n-1$.
    Thus we have $x_{i+m} + x_i \equiv_n -k$, and so $x_{i+2m} + x_{i+m} \equiv_n -k$, which implies that $x_i \equiv_n x_{i+m}$ for all $0 \leq i \leq n-1$, and hence $m=0$ or $m = n/2$.
    But if $m = 0$, then from the evaluation of $A_2$ above, no strings will be fixed, and so we must have $m = n/2$ and $n$ must be even.
    
    Hence we have that $x_i + x_{i+ \frac{n}{2}} \equiv_n -k$.
    Note that $k$ cannot be even, since, as $n$ is even, $-k$ would also be even, and so we would have $x_i = -k/2 = x_{i+\frac{n}{2}}$, which cannot be the case since $c^\frac{n}{2}$ fixes no points.
    
    Thus $k$ must be odd, and so for all $n/2$ odd choices of $k$ we have $(n/2)!$ bijections
    \begin{align*}
        P_{\mathsf{ind}} &= \set{ \set{0, \frac{n}{2}}, \set{1, \frac{n}{2} + 1}, \dots, \set{\frac{n}{2}-2, n-2}, \set{\frac{n}{2}-1, n-1} } \\
        \updownarrow & \\
        P^{(k)}_{\mathsf{lab}} &= \set{ \set{0, -k}, \set{1, -(k+1)}, \dots, \set{ \frac{k-1}{2}, \frac{k+1}{2}} },
    \end{align*}
    each of which affords $2^{(n/2)}$ fixed strings.
    Therefore we obtain
        \[
            B_2 = 
                \begin{cases}
                \frac{n}{2}n!!, &\textnormal{if $n$ is even;}\\
                0, &\textnormal{if $n$ is odd.}
                \end{cases}
        \]
    \item
    Notice that $(c^mv, r^k)$ fixes $(x_0 x_1 \dots x_{n-1})$ if and only if $x_{-(i+1) + m} + k \equiv_n x_i$ for all $0 \leq i \leq n-1$, which is the case if and only if $x_{-(i+1)} + k \equiv_n x_{i-m}$ for all $x_i$.
    There are three cases.
        \begin{enumerate}
            \item If $n$ is odd, then for all values of $m$ there exists a unique $0 \leq a \leq n-1$ such that $-(a+1) \equiv_n a-m$, and hence $x_{-(a+1)} = x_{a-m}$ and consequently $k=0$.
            But, since all other entries of the string are moved, it follows that no odd-length strings will be fixed by $(c^mv, r^k)$.
            \item If $n$ is even and $m$ is odd, then there exist exactly two indices $0 \leq a, b \leq n-1$ such that $-(a+1) \equiv_n a-m$ and $-(b+1) \equiv_n b-m$, and hence $x_{-(a+1)} = x_{a-m}$ and $x_{-(b+1)} = x_{b-m}$.
            Then, just as above, we have $k=0$ and thus no strings will be fixed.
            \item If both $n$ and $m$ are even, then, since there must be some $x_i = 0$, each of $n/2$ choices of $m$ will fully determine the value of $k$.
            Hence, in the same manner as before, for each $m$ we consider the $(n/2)!$ bijections between
                \begin{align*}
                    P^{(m)}_{\mathsf{ind}} &= \set{ \set{0, n-1+m}, \set{1, n-2+m}, \dots, \set{\frac{n}{2}-1, \frac{n}{2}+m} } \\
                    \updownarrow & \\
                    P^{(m)}_{\mathsf{lab}} &= \set{ \set{0, -k}, \set{1, 1-k}, \dots, \set{\frac{n}{2}-1, \frac{n}{2}-1 -k} };
                \end{align*}
            each of which, as before, yields $2^{\frac{n}{2}}$ fixed strings, for a total of $n!!$ fixed strings for each value of $m$.
        \end{enumerate}
    Hence we obtain
        \[
            B_3 = 
                \begin{cases}
                \frac{n}{2}n!!, &\textnormal{if $n$ is even;}\\
                0, &\textnormal{if $n$ is odd.}
                \end{cases}
        \]
    \item
    Notice that $(c^m v, sr^k)$ fixes $(x_0 x_1 \dots x_{n-1})$ if and only if $-(x_{-(i+1) +m} + k) \equiv_n x_i$ for all $0 \leq i \leq n-1$.
    There are three cases.
    \begin{enumerate}
        \item If $n$ is odd, then it is tedious but not difficult to see that an analogous argument to the evaluation of $A_4$ in the proof of Theorem \ref{thmsp} applies for all $n$ values of $m$, yielding a total of $n^2(n-1)!!$ fixed strings.
        \item If $n, m$ are both even, then it is again not difficult to see that an analogous argument to the evaluation of $A_4$ in the proof of Theorem \ref{thmsp} applies for all $n/2$ even values of $m$.
        That is, each even choice of $m$ allows for $n/2$ odd values of $k$, each of which affords $n!!$ fixed strings.
        \item If $n$ is even and $m$ is odd, then there are exactly two indices $0 \leq a, b \leq n-1$ such that $2a \equiv_n 2b \equiv_n m+1$.
        It follows that $m-a-1 \equiv_n a$ and $m-b-1 \equiv_n b$, and so  $-x_a - k \equiv_n x_a$ and $-x_b - k \equiv_n x_b$.
        Hence $2x_a \equiv_n 2x_b \equiv_n -k$, and so $k$ must be even.
        
        Now, for particular values of $m, a, b$, there are $n/2$ possible even values of $k$.
        Each choice of $k$ fixes the values of $\set{x_a, x_b}$, which may be ordered in two ways; and, by the same methods as before, both choices of ordering afford $(n-2)!!$ fixed strings. 
        Hence, in this case we have $(n/2) \cdot (n/2) \cdot 2 \cdot (n-2)!! = (n/2)n!!$ fixed strings in total.
    \end{enumerate}
    Thus we obtain
        \[
            B_4 = 
                \begin{cases}
                (\frac{n}{2}+1)\frac{n}{2}n!!, &\textnormal{if $n$ is even;}\\
                n^2(n-1)!!, &\textnormal{if $n$ is odd.}
                \end{cases}
        \]
\end{enumerate}

Having evaluated each of these sums, the desired theorem now follows.
\end{proof}

As before, since $\Gamma(n, C, \equiv_E)$ is a subgroup of $\Gamma(n, C, \equiv_S)$, the number of equivalence classes of $C_n$ under $\equiv_E$ follows as an easy corollary.

\begin{theorem} \label{thmec}
    Let $n \geq 3$ be an integer. Then the number of equivalence classes of $C_n$ under $\equiv_E$ is given by
        \[
            |C_n / \equiv_E| = \frac{1}{2n^{2}} \left[ \sum_{d | n} \left( \left(\phi\left(\frac{n}{d} \right) \right)^2  \left(\frac{n}{d}\right)^d  d! \right) +
                \begin{cases}
                \frac{n}{2}n!!, & \textnormal{if $n$ is even;} \\
                0, & \textnormal{if $n$ is odd.}
                \end{cases}
            \right]    
        \]
    \end{theorem}

Finally, it is worth noting a small corollary to the above theorems. 
Since $\phi(p) = (p-1)$ for any prime $p$, we have the following.

\begin{corollary}
    Let $p>2$ be prime. Then 
        \[
            |C_p / \equiv_S| = \frac{1}{4p}\left[ (p-1)^{2} + p(p-1)!! + (p-1)! \right],
        \]
    and
        \[
            |C_p / \equiv_E| = \frac{1}{2p}[(p-1)^{2} + (p-1)!].
        \]
\end{corollary}

\section{Further Remarks}

Here we note some interesting connections which the authors noticed over the course of writing this paper. 
After completing the enumeration of $P_n / \equiv_S$, we discovered that there are exactly as many of them as there are \textit{tone rows} in $n$-tone music---the enumeration of which may be found in a paper of Reiner \cite{rein}.
The corresponding OEIS sequence is sequence \seqnum{A099030}---which, as has been noted, is identical to sequence \seqnum{A089066}.

Further, for reasons which should be clear, there are exactly as many classes in $C_n / \equiv_S$ as there are classes of similar $n$-gons (that is, classes of $n$-gons which are equivalent up to rotations and reflections). 
These classes---as well as the analogous case of $n$-gons equivalent up to rotations only---were enumerated in a 1960 paper of Golomb and Welch \cite{golomb}.
As such, this paper provides an alternative proof of their result. 
The corresponding OEIS sequences are \seqnum{A000940} and \seqnum{A000939}, respectively.
It should also be noted that the evaluation of $B_1$ in the proof of Theorem \ref{thmsc} is in large part an adaptation of an argument of Moser \cite{moser}.
In particular, the sum of Euler $\phi$ terms which makes an appearance in this paper as well as in the paper of Golomb and Welch \cite{golomb} is the same as that which appears in the case of $a=1$ in Moser's paper \cite{moser}. 
This connection is (as far as the authors are aware) not yet noted anywhere.

\section{Acknowledgements}

The authors would like to thank Chris Kottke for providing valuable feedback on several early versions this paper.
We would also like to thank the anonymous reviewer for his or her constructive suggestions, many of which were incorporated into this final version.
Finally, the first author would like to thank Nika Sigua for his role in the initial discovery of the problems discussed in this paper.


\bigskip
\hrule
\bigskip

\noindent 2010 {\it Mathematics Subject Classification}:
Primary 05C30; Secondary 05E18.

\noindent \emph{Keywords:} Hamiltonian Path, equivalence class, group action, Burnside's lemma.

\bigskip
\hrule
\bigskip

\end{document}